\documentclass[12pt]{article}
\usepackage{graphicx}
\usepackage{amssymb,url, amsthm, amsmath, verbatim}
\newtheorem{theorem}{\noindent Theorem}
\newtheorem{lemma}{\noindent Lemma}
\newtheorem{definition}{\noindent Definition}
\newtheorem{corollary}{\noindent Corollary}

\begin{document}
\begin{center}{}{\bf INVARIANT MEASURES FOR
THE CONTINUAL CARTAN SUBGROUP}

\bigskip

{}\bf\textbf{A.~M.~VERSHIK}
\end{center}
\bigskip
\rightline{\it To Professor Malliavin with deep respect}
\bigskip

\begin{abstract}
We construct and study the one-parameter semigroup of
$\sigma$-finite measures ${\cal L}^{\theta}$, $\theta>0$, on the
space of Schwartz distributions that have an infinite-dimensional
abelian group of linear symmetries; this group is a continual analog
of the classical Cartan subgroup of diagonal positive matrices of
the group $SL(n,R)$. The parameter $\theta$ is the degree of
homogeneity with respect to homotheties of the space, we prove
uniqueness theorem for measures with given degree of homogeneity,
and call the measure with degree of homogeneity equal to one the
infinite-dimensional Lebesgue measure
 $\cal L$. The structure of these measures is very closely related
 to the so-called Poisson--Dirichlet measures $PD(\theta)$, and to the well-known
 gamma process. The nontrivial properties of the Lebesgue measure
 are related to the superstructure of the measure $PD(1)$, which is
 called the conic Poisson--Dirichlet measure -- $CPD$. This is the most
 interesting $\sigma$-finite measure on the set of positive convergent
 monotonic real series.
\end{abstract}

\tableofcontents

\section{Introduction} The Lebesgue measure on the cone of positive
vectors of ${\mathbb R}^n $ can be described as the unique (up to a
multiplicative constant) $\sigma$-finite measure that is invariant
with respect to the group $SDiag_+(n,\mathbb R)$  (the positive part of
the Cartan subgroup) and homogeneous of
degree $n$ with respect to homotheties.
What is the infinite-dimensional analog of this fact?
First of all, what is the continual analog of the Cartan subgroup?
And does there exist,  in an infinite-dimensional space, a
measure that is invariant with respect to the Cartan subgroup and has
a finite degree of homogeneity with respect to homotheties?

The goal of this paper is to introduce a family of $\sigma$-finite
measures in the space of Schwartz distributions on the interval (or
on a manifold) that is invariant with respect to a continual group
of multiplicators and has a finite degree of homogeneity with
respect to homotheties. The measure with degree of homogeneity equal
to one is called the {\it infinite-dimensional Lebesgue measure}.

More exactly, we consider $\sigma$-finite Borel measures, in the
space of Schwartz distributions on a manifold $X$ with measure $m$,
that are invariant under the action of the group ${\cal M}_0$ of
operators of multiplication by positive test functions $a(\cdot)$
satisfying the condition $$\int_X \ln a(x)dx=0$$. This group is a
direct continual analog of the positive Cartan subgroup $SDiag_+(n,
\mathbb R)$ of the group $SL(n,\mathbb R)$. The above condition is
an analog of the condition $\det A=1$ for matrices.

It turns out that this invariance implies another invariance: these
measures are invariant under the action of the group of
measure-preserving transformations of the manifold (this is an
analog of the invariance of the ordinary Lebesgue measure under the
action of the symmetric group of permutations of coordinates); this
group plays role of Weyl group in $SL(n,\mathbb R)$.

In fact, we define a remarkable one-parameter family of
$\sigma$-finite measures ${\cal L}^{\theta}$,
$\theta>0 $, enjoying this property, and prove a uniqueness theorem which
says that ${\cal L}^{\theta}$ is the unique ergodic, positive, $\sigma$-finite
measure on the cone of positive Schwartz distributions that is
finite on compact subsets, satisfies
this invariance property, and has a fixed degree of homogeneity.

The  measures  ${\cal L}^{\theta}$  are closely related to
the so-called Poisson--Dirichlet measures $PD(\theta)$ on the
infinite-dimensional simplex and, consequently, to many combinatorial
problems (see \cite{V}).

One of the main applications of our measures is in the
representation theory of current groups. The group of symmetries of
the Lebesgue measure $\cal L$ (the Cartan subgroup) allows us to use
the space $L^2({\cal L})$ for constructing a realization of
representations of the current group $G^X$, where $G$ is a
semisimple Lie group of rank $1$. This measure was discovered in a
series of papers by (\cite{VGG1,VGG,GGV}), and later
(\cite{VG,VG1}), and applied for constructing a new model of these
representations which, instead of the Fock space with
infinite-dimensional range of the one-particle subspace, uses
$L^2(\cal L)$. More generally, this gives a  {\it new look on the
construction of the continuous tensor product of Hilbert spaces}. We
can say that our construction is a contribution to the continuing
discussion of what is integration in infinite-dimensional spaces.

Another aspect, which is discussed in the papers \cite{VY, TVY},
is a deep
parallelism between the Gaussian measure (white noise) and our
Lebesgue measure. These two processes, the first one being a generalized
stochastic process and the second one being a quasi-stochastic process
(the term ``quasi-stochastic'' means that instead of a
probability measure on the space of the process we use a
$\sigma$-finite measure)
correspond to the two endpoints of the interval
$[0,2]$ which is the set of parameters of stable L\'evy processes (or stable
distributions on the real line). The point $2$ of the interval
$(0,2]$
corresponds to the white noise (or to the Wiener process, if one
prefers to consider a process with independent increments instead of
a generalized process with independent values); the point $0$, as the
parameter of a stable distribution, has no direct sense: {\it at
this point, we must consider the derivative of the characteristic
functional}, and this gives our Lebesgue measure (see \cite{TVY} for the
so-called notion
of zero-stable law and zero-stable distribution). The
Gaussian measure (white noise) corresponding to the point $2$ is invariant under the
action of the infinite-dimensional orthogonal group; at the point
$0$, we have the invariance under the infinite-dimensional (abelian)
linear group of multiplicators. The interval $(0,2)$ can be
regarded as a deformation of one measure into the other one with unknown
nonlinear group of symmetries.

This paper is closely related to our previous papers \cite{V,V1}
and devoted to the same subject. In this paper, we put an end
to the discussion concerning
approximation of these measures
with finite-dimensional invariant measures on the hyperspheres, which was
started in \cite{V} (see also \cite{V1,VG2}). In that paper we
insisted on an analogy between the situation with approximation of
the infinite-dimensional Gaussian measure and the infinite-dimensional
Lebesgue measure. In this paper we explain why there is such a drastic
difference between the case of the white noise (the Maxwell--Poincar\'e lemma
about the limit of measures on Euclidean spheres) and our case with
noncompact homogeneous spaces (hyperspheres). The key difference is
the difference between the $\sigma$-finiteness and finiteness of the
orbital measures, or, shortly, between the noncompactness and compactness
of the group of finite-dimensional symmetries (the Cartan and
orthogonal groups). In a more physical language, this means the absence
of equivalence between the grand and small canonical ensembles in the
Lebesgue case, in contrast to the equivalence of both ensembles in
the Gaussian case, see Section~4.

We want to emphasize the similarity between the construction of the
measures ${\cal L}^\theta$ and
the classical method of constructing stochastic
L\'evy processes as Gelfand--It\^o processes. It is based on
infinitely divisible measures and the corresponding characteristic
functionals (in the theory of stochastic processes, there is an
equivalent language of ``random measures'').  Unfortunately, there is
no (at least to the author's knowledge) general theory of
``$\sigma$-finite L\'evy processes'' and no L\'evy--Khintchin formula for the
Fourier (or Laplace) transform of infinitely divisible
$\sigma$-finite measures on the real line and on ${\Bbb R}^n$.
Undoubtedly, the needed definition must be very similar to the
classical one, and they must be unified in a general theory of finite
and $\sigma$-finite L\'evy processes. But since there is no such a theory,
we use special considerations in our construction, see Section~2.

Let us briefly mention the history of what we called
the infinite-dimensional Lebesgue measure. It seems that this measure first
appeared in the papers \cite{VGG,GGV}, as a measure whose
characteristic functional is the restriction of the canonical state on the
group $SL(2, \Bbb R)$ to the unipotent subgroup. Later it became
clear that there is a link between the gamma process and the
Lebesgue measure: the latter is absolutely continuous with respect
to the law of the gamma process, see \cite{TVY} and references therein.
This link was discussed in connection with the Markov--Krein
transformation \cite{TV}, stable laws \cite{V}, isomorphisms with the
Fock space \cite{VTs}. The connections between the representation theory of
current groups and the Lebesgue measure was recently considered in
the joint paper of the author with M.~I.~Graev \cite{VG}.

Here are brief comments on the contents of the paper. In the
second section we recall the definitions and basic properties of
some important measures and processes
(Dirichlet, gamma) and introduce the measures ${\cal
L}^{\theta}$, first as generalized processes or random measures.
We use the link between these measures and the gamma processes. It is
possible to generalize the classical theory of
extension of cylinder measures in linear spaces, including the
definition of Gelfand--It\^o generalized processes,
to the $\sigma$-finite case. The reference to
the properties of the gamma processes immediately allows us to establish
that the support of our measures ${\cal L}^{\theta}$ is the set of
discrete positive finite measures on the manifold (the
simplest type of Schwartz distributions)
$$\xi= \sum_k c_k\cdot \delta_{x_k}, \quad c_k>0, \;x_k \in X,\; k=1,2,
\dots.$$

 In the third section we describe the properties of the
measures  ${\cal L}^{\theta}$
and their connection with the Poisson--Dirichlet measures $PD(\theta)$.
We define the Laplace transform of  ${\cal L}^{\theta}$:
$$\Psi_{{\cal
L}_{\theta}}(f)\equiv \Psi (f) =e^{-\theta \int \ln f(x)dm(x)},
\quad m(X)=\theta>0.$$
In the previous paper \cite{V} we took the formula for the Laplace transform
as a definition of these measures. We establish the invariance of
our measures with respect to the continual Cartan group ${\cal M}_0$ (see
above) and prove the uniqueness of an invariant measure with a given degree of
homogeneity.

 There are many papers about the Poisson--Dirichlet measures
$PD(\theta)$ (see \cite{King,Arr} and the references therein, and also the earlier
papers \cite{VSh}). The measure $PD(\theta)$ is concentrated on the
infinite-dimensional simplex of monotonic positive series with sum equal to one. It
is covered by
a measure on the larger simplex of all positive
series with sum equal to one, which is
usually called the GEM-measure
(we do not mention it in the paper). It is worth mentioning that the GEM-measure was
perhaps first introduced, in a different form, by W.~Eberlein \cite{Eber}
(even for nonpositive series).

We also consider the conic measures $CPD(\theta)$ and prove the
characteristic properties of these measures. The measure
$CPD(\theta)$ is maybe the most interesting object. In a sense, it
is more natural than the measure $PD(\theta)$. This measure is
concentrated on the cone of convergent monotonic positive series and
has a large group of symmetries, which comes from the symmetries of
the measure ${\cal L}^{\theta}$. It is interesting to obtain a
characterization of the measure $CPD(\theta)$ directly.

The fourth section is devoted to the approximation theory of the
measures ${\cal L}^{\theta}$ and comparison with white noise - we
show that there  is the new phenomenon of the absence of
approximation and absence of equivalence of grand and small
canonical ensembles in the case of the Lebesgue measure. We
introduce in the forth section the new function ($L$) (see also
\cite{VV}) which is similar to free energy and must have important
role in this theory. We only mention the two important
generalizations of this measure: to absolutely convergent
(nonpositive) series and to two-sided series; this will be done
later.

\section{Old and new definitions; construction of the measures ${\cal L}^{\theta}$}

\subsection{The spaces}

 Let $(X,m)$ be a manifold with a fixed continuous finite Borel
measure $m$. In fact, in what follows we may assume that $(X,m)$ is the
interval $[0,\theta]$, $\theta>0$, or the unit circle, with the Lebesgue measure
$m$: this case does not essentially differ from the case of a general manifold
and even an arbitrary standard Borel space. The measure $m$ is
continuous, nonnegative, and finite, but not necessarily normalized; we
put $m(X)=\theta >0$, so that ${\tilde m}(\cdot) =\frac {m(\cdot)}{m(X)}$ is a
normalized measure. We will consider measures as elements of
the space
$$l^1(X)=\{\eta=\sum_k c_k\cdot \delta_{x_k}:  \quad \sum_{k=1}^{\infty}
|c_k|<\infty, \; x_k\in X,\; k=1,2, \dots \}$$
 of discrete (countable) finite signed
measures, and in the most part of the paper we deal only with the
positive cone
$$l_+^1(X)=\{\eta=\sum_{k=1}^{\infty} c_k\cdot \delta_{x_k}: \quad
\sum_{k=1}^{\infty} c_k < \infty, \; c_k>0, \; x_k\in X, \;
k=1,2,\dots \}$$
 of the linear space $l^1(X)$;
we also need to use its convex subset $s^1(X) \subset \l^1_+(X)$ defined as
$$s^1(X)=\{\eta=\sum_{k=1}^{\infty} c_k\cdot \delta_{x_k}: \quad
\sum_{k=1}^{\infty} c_k =1, \; c_k>0, \; x_k\in X, \;
k=1,2,\dots \}.$$

We have the following obvious decomposition: $l^1_+(X)\setminus
\{\textbf{0}\}=s^1(X)\times (0,\infty)$.

For some reason, we need to consider the space of Schwartz
distributions (generalized functions) $D(X)$, which contains  $l^1(X)$,
but we can avoid this by using an {\it a priori} knowledge about the structure
of our measures.

 We will equip the space $l^1(X)$ and the cone $l^1_+(X)$ with the weak topology
 that comes from the natural duality between the space $l^1(X)$ and the space of bounded
 continuous, or even measurable, functions on the space $X$:
 $$<f,\eta>=\sum c_k \cdot f(x_k), \quad f\in C(X), \; \eta=
  \sum_{k=1}^{\infty} c_k\cdot \delta_{x_k}\in l^1(X).$$
Functions $f$ from the dual space will be called test functions.
 We fix the Borel structure on the cone $l^1_+(X)$.\footnote{This structure
 does not coincide with the one that can appear if one considers $l^1(x)$ as a nonseparable
 Banach space.}
 Our goal is to define and study distinguished positive Borel measures on the
 cone $l^1_+(X)$.

\subsection{The Dirichlet and gamma distributions}
We recall the definitions of several finite-dimensional measures.
 Let $\Sigma_n=\{(x_1,\dots, x_n): \Sigma_{i=1}^n x_i=1,\, x_i\geq 0,\, i=1,
\dots, n\}$ be the standard $(n-1)$-dimensional simplex. The
 Dirichlet distribution $D^{\bar \theta}$ on the simplex $\Sigma_n$ with parameter ${\bar
\theta}=(\theta_1, \dots, \theta_n)$ is
the probability measure whose density with respect to the Lebesgue
measure\footnote{When we speak about the Lebesgue measure on $\Bbb R$
or ${\Bbb R}^n$, we always mean that it is normalized so that the
measure of the unit interval or cube is equal to $1$.} is given by the formula
$$\Gamma(\theta)\cdot
\prod_{i=1}^n \frac{x_i^{\theta_i-1}}{\Gamma(\theta_i)},$$
where $\theta=\sum_1^n
\theta_i$ (see \cite{King}).

The well-known gamma measure
on the half-line $\Bbb R_+$ with
parameter $\theta$  has the following
density with respect to the Lebesgue measure:
$$
\lambda^{\theta}(x)=\frac{x^{\theta -1}e^{-x}}{\Gamma(\theta)}, \quad x>0.$$

The gamma measure on the orthant ${\Bbb R}_+^n$
with parameter ${\bar \theta}=(\theta_1, \dots, \theta_n)$
is defined as the product of $n$ one-dimensional gamma measures with parameters
$\theta_1, \dots, \theta_n$ and
has the density
$$\lambda^{\bar \theta}(x)=
\prod_{i=1}^n\frac{x_i^{\theta_i-1}e^{-x_i}}{\Gamma(\theta_i)}.$$

The Dirichlet distribution with parameter $\bar \theta$ is nothing
more than the projection of  the gamma measure
with the same parameter from the orthant to the simplex. The gamma distribution $\lambda^{\theta}$
on the half-line is an infinitely divisible probability measure which
generates a L\'evy process called the  gamma process.

The following simple fact is a consequence of the definitions, but
it plays a very important role. Let us represent the orthant ${\Bbb R}^n_+$ as the
product of the simplex and the ray $(0,\infty)$; then, with respect to this
decomposition, we have the following decomposition of probability measures:

\begin{lemma}
Let ${\bar \theta}=(\theta_1, \dots ,\theta_n)$, $\theta=\theta_1+
\dots +\theta_n$. Then
$$\lambda^{\bar \theta}=D^{\bar \theta}\times \lambda^{\theta}.$$
\end{lemma}

This is a characteristic property of the gamma distributions. The same
is true for the law of the gamma subordinator (see below): this
measure is the direct product of its projection onto the simplex and the
gamma distribution on the half-line, and this fact characterizes
the gamma processes among all L\'evy processes (Lukacs' theorem, see
\cite{TVY}).

\subsection{Weak distributions, random measures, and
generalized processes}

Now we introduce infinite-dimensional
measures.  In the theory of stochastic processes, real-valued
L\'evy processes, as well as other types of stochastic processes, are
traditionally defined with the help of so-called ``random measures,'' i.e.,
measurable maps from the $\sigma$-field  of measurable sets of the
parametric space to the space of measurable functions (random
variables) with real values (see \cite{D,King} and references
therein). We interpret the notion of a
random measure  in the spirit of Gelfand--It\^o  generalized stochastic
processes. Namely, the general definition of
a generalized process, or weak distribution, is as follows. It is
a self-consistent  system of finite-dimensional distributions
defined for each finite collection of linear functionals
from a total set of linear functionals on the space $D(X)$ (i.e.,
a set of functionals  whose linear closed hull is the whole space),  for example,
the set of all characteristic functions of measurable subsets of $X$. In our case, we
will associate with an arbitrary finite measurable partition
of the parametric space $X$ a finite-dimensional distribution in $\mathbb R^n$
that is a self-consistent system of finite-dimensional distributions
and is continuous with respect to the partitions. This is just
a random additive  measure in the usual sense
from the point of view of the theory of generalized processes
and weak distributions on the space of distributions $D(X)$.
According to Minlos' theorem, such a distribution determines a true probability measure
 on $D(X)$ (the law of the process), see, e.g., \cite{GV}. We
do not know if such a theory exists for the case of $\sigma$-finite
 measures, so we use  direct arguments
 for the description of $\sigma$-finite measures
(see \cite{GGV,TVY}).

 \subsection{Three self-consistent systems of distributions
and the definition of the Lebesgue measure}

We will simultaneously give the definitions of three
one-parameter families of ``random measures,'' or generalized
stochastic (in the third case, quasi-stochastic) processes. They
determine the corresponding families of measures on the cone
$l^1_+(X)$. Two of them are classical and well-known:
the  {\it Dirichlet process} and its measure $D^{\theta}$; and the
{\it gamma process} and
 its measure (law) $\Lambda^{\theta}$. The third one is a ``quasi-stochastic''
 process, which produces a $\sigma$-finite measure ${\cal L}^{\theta}$;
  in particular, in this way we obtain the {\it Lebesgue measure} $\cal L$.

Fix a positive number $\theta>0$ as a parameter and choose any
vector $\bar \theta=(\theta_1, \theta_2,\dots, \theta_n)$ with
$\sum_i \theta_i= \theta=m(X)$.

Assume that we fix a finite sub-$\sigma$-field of the $\sigma$-field of
 all measurable sets of $X$, or simply a partition
 $\xi_{\bar \theta}$,
 $X=\bigcup_{i=1}^n E_i$, of the space $X$ into measurable sets
 $E_i$, $i=1,\dots, n$,
 with measures $m(E_i)=\theta_i$, $i=1, \dots, n$.
 The finite-dimensional
 distributions corresponding to this partition have the following
 densities with respect to the Lebesgue measure:

\smallskip
 1. {\it The Dirichlet process.} The distribution corresponding
 to the partition $\xi_{\bar
\theta}$ is the Dirichlet distribution with parameter ${\bar
\theta}=(\theta_1, \dots, \theta_n)$:
 $$D^{\bar \theta}(x_1,x_2,\dots,x_n)=\frac{\Gamma(\theta)}{\prod_1^n\Gamma(\theta_i)}
 \prod_1^n x_i^{\theta_i-1}dx_i.$$
This is a distribution on the simplex.

 \smallskip
 2. {\it The gamma process.} The distribution corresponding
 to the partition $\xi_{\bar \theta}$
 is the gamma distribution:
 $$\lambda^{\bar \theta}(x_1,x_2,\dots,x_n)=\prod_{i=1}^n \frac{x_i^{\theta_i-1}}{\Gamma(\theta_i)}
 \cdot e^{-\theta x_i}.$$

\smallskip
 In both cases~1 and~2, the definitions yield true probability processes;
in case~2 we obtain a L\'evy process.

\smallskip
 3. {\it The Lebesgue case.} The distribution corresponding
 to the partition $\xi_{\bar \theta}$ is the
 $\sigma$-finite measure on the orthant ${\Bbb R}^n_+$ with the following density:
 $$L_{\bar\theta}(x_1,x_2,\dots,x_n)=\prod_{i=1}^n
\frac{{x_i}^{\theta_i-1}}{\Gamma(\theta_i)}.$$

This density is the same, up to a scalar, as in case 1, but here we
consider it not on the simplex, but on the orthant ${\Bbb R}^n_+$.
Note that in cases~2 and~3, the density is the product of
one-dimensional densities, so we can say that in case~3 we also have
a {\it ``quasi-stochastic'' process with independent
values};\footnote{With an appropriate definition of independence.} and the
family $\{L_{\theta}\}$ of one-dimensional distributions with
densities $dL_{\theta}=\frac{x^{\theta-1}}{\Gamma(\theta)}dx$ is a
{\it multiplicative semigroup} with parameter $\theta\in (0,\infty)$
with respect to the convolution on the half-line $(0,\infty)$:
 $L_{\theta_1}*L_{\theta_2}=L_{\theta_1+\theta_2}$ (the $\delta$-measure at $0$
 can be regarded as the identity element of the semigroup).

The self-consistency and continuity of all the systems above can be
checked directly. In all three examples there is one parameter,
$\theta=m(X)$. The first two cases are well-known, and we can conclude
the existence of true probability measures: the law of the
Dirichlet process $D^{\theta}$ and the law of the gamma process $\Lambda^{\theta}$.
It is well known that the Dirichlet
measure is concentrated on the set $s^1(X)$ of discrete probability
measures, and the gamma measure is concentrated on the cone $l^1_+(X)$ of
finite
discrete measures. For case~3, we prove the following theorem-definition.

\begin{theorem}
{\rm1.} For every $\theta>0$, the system $\{L_{\bar\theta}\}$\footnote{Here
$\theta$ is fixed and $\bar \theta $ ranges over
all vectors $\bar\theta$ with the given sum $\theta$ of the
coordinates.} is a self-consistent system of
finite-dimensional
distributions and defines a $\sigma$-finite measure ${\cal L}^{\theta}$
on the space of Schwartz distributions $D(X)$.

{\rm2.} All the measures ${\cal L}^{\theta}$ are concentrated on the cone
$l^1_+(X)$ and take finite values on compact subsets of
$l^1_+(X)$.

{\rm3.} The measure ${\cal L}^{\theta}$ is a $\sigma$-finite measure absolutely
continuous with respect to the measure $\Lambda^{\theta}$ with a nonintegrable
density
  $$\frac{d{{\cal L}^{\theta}}}{d\Lambda^{\theta}}\bigg(\sum_k c_k\cdot\delta_{x_k}\bigg)=e^{\sum_k c_k}.$$

{\rm4.} In the decomposition $l^1_+(X) \setminus \{{\bf 0}\}=s^1(X) \times
(0,\infty)$ above, the measure ${\cal L}^{\theta}$ is the direct
product of the Dirichlet measure and the measure $L^{\theta}$ on
$(0,\infty)$:
 ${\cal L}^{\theta}=D^{\theta}\times L^{\theta}$; in particular,
for the Lebesgue measure we have ${\cal L}=D^1\times L$, where $L$ is
the Lebesgue measure on $(0,\infty)$.

{\rm 5.} The family of measures $\{{\cal L}^{\theta}:\theta>0\}$ with
additive parameter $\theta$ is a semigroup of measures on $l^1_+(X)$
with respect to the convolution.
\end{theorem}

\begin{proof}
1, 2, 3. Consider the $\sigma$-finite measure that has the
following density, regarded as a function on the cone $l^1_+(X)$, with respect
to the law $\Lambda^{\theta}$ of the gamma process:
$$\frac{d{{\cal L}^{\theta}}}{d\Lambda^{\theta}}\bigg(\sum_k c_k\cdot\delta_{x_k}\bigg)=e^{\sum_k c_k}$$
(this function is well defined for $\Lambda^{\theta}$-almost all
points). Denote this measure by ${\cal L}^{\theta}$. It is clear from
the definition that this measure produces the same joint distributions
for partitions $\xi_{\bar \theta}$ as in the
definition of $\Lambda^{\theta}$.
This gives the existence of a
measure with the given projections. The uniqueness of a $\sigma$-finite
measure with the given proper (finite)
distributions
follows from the fact that the set of characteristic functions is
total as a set of functionals on the space $D(X)$\footnote{In the
case of $\sigma$-finite measures, it may happen that the projections to
some (or all) finite-dimensional factors are totally
infinite measures.}. The finiteness of the values of the measure ${\cal L}^{\theta}$
on compact subsets follows from the boundedness of the density
on compact subsets.

4. This decomposition is, of course, a
consequence of the decomposition $\Lambda^{\theta}=D^{\theta}\times
\lambda^{\theta}$ mentioned above.

5. This is an obvious consequence of the multiplicative formula for the
corresponding finite-dimensional distributions.
\end{proof}

{\it We call ${\cal L}^1\equiv \cal L $ the
infinite-dimensional Lebesgue measure on the cone $l^1_+(X)$}. Claim 4 of
the theorem says that the Lebesgue measure $\cal L$ is a {\it ``conic
superstructure'' of the Dirichlet measure $D^1$}.

\section{Further properties of the measures ${\cal L}^{\theta}$}

Consider the properties of the $\sigma$-finite measures ${\cal
L}^{\theta}$.

\subsection{The Poisson--Dirichlet measure and the decomposition of ${\cal L}^{\theta}$.}

First of all, the measures ${\cal L}^{\theta}$, as well as the gamma
measures $\Lambda^{\theta}$, are concentrated on the cone
$l^1_+(X)$: ${\cal L}^ {\theta}(D(X)\setminus l^1_+(X))=0$; and
since they are equivalent to the gamma measures, we can obtain
properties that are valid for ${\cal L}^{\theta}$-almost all points.

 Recall some definitions concerning the Poisson--Dirichlet
 measures $PD(\theta)$
(see \cite{King}). Consider the infinite-dimensional simplices of
positive convergent series with sum equal to 1,
$$\Sigma_1=\{\{c_k\}_{k=1}^{\infty}: \sum_k c_k=1\},$$
and positive monotonic convergent series with sum equal to 1,
$$\bar\Sigma_1=\{\{c_k\}_{k=1}^{\infty}: c_1\geq c_2 \geq \dots, \; \sum_k c_k=1\}.$$

The shortest definition of the measures $PD(\theta)$ on the simplex
$\bar\Sigma_1$ uses a map from the infinite-dimensional cube
$$Q^{\infty}=\{\{y_i\}_i: 0\leq y_i\leq 1,\; i=1,2, \dots \}$$ to the
simplex $\bar\Sigma_1$. First we send the cube $Q^{\infty}$ to the
simplex $\Sigma_1$ by  the map $T$ given by the formula
$$T(\{y_k\}_{k=1}^{\infty})=\{c_i\equiv y_i\prod_{j=1}^{i-1}(1-y_j)\}_{i=1}^{\infty},$$
and then use the map $M$ that orders the elements of  the series $\{c_i\}$ by
decreasing. Consider the Bernoulli measure $\mu_{\theta}$, $\theta>0$,
on the cube $Q^{\infty}$ that is the infinite power of the measure on$[0,1]$ with the density $\theta x^{\theta-1}$.

\begin{definition}
The Poisson--Dirichlet measure $PD(\theta)$ is the image of the measure
$\mu_{\theta}$ under the map $MT:Q^{\infty}\rightarrow \bar\Sigma_1$.
\end{definition}

There are many papers about the measures $PD(\theta)$
(e.g., \cite{Arr,VSh,YorP,King}). The deep facts about the  structure of the
measure $PD(1)$ that were presented in \cite{VSh} still have not
found enough applications.

The following theorem reduces the study of the structure of the
Dirichlet measure $D^\theta$ and the law of the gamma process with
parameter $\theta$ to the study of the Poisson--Dirichlet measure $PD(\theta)$.
Let us decompose the continual simplex $s^1(X)$ as follows. Every element
$\xi \in s^1(X)$ is a series
 $\xi=\sum_k c_k \cdot \delta_{x_k}$, $c_k>0$, $c_1\geq c_2\geq \dots$,
$\sum_k c_k=1$.
The correspondence
$$J:\xi\mapsto \{(c_k,x_k)\}_{k=1}^{\infty},$$
which sends $s^1(X)$ to $\bar \Sigma_1
\times \prod_{k=1}^{\infty} X$, is a bijection when restricted to the set
of series $\xi$ with distinct coefficients $c_k$ (if some $c_k$
has a multiplicity in $\xi$, then the corresponding (finite) set of
points $x_i$ with this coefficient can be enumerated in an arbitrary
manner). It is obvious that the set of discrete measures $\xi=\sum_k
c_k \cdot \delta_{x_k}$ with distinct coefficients ($c_i \ne c_k$ for  $i\ne
k$) is of full measure $\Lambda^{\theta}$; consequently, we can say
that the map $J$ is a bijection $\bmod 0$ (almost everywhere). The
following fact is known.

\begin{lemma}
The bijection $J$ is a measure-preserving map from the measure space
$(s^1(X), D^{\theta})$ to the product of measure spaces $({\bar
\Sigma_1}, PD(\theta)) \times (X^{\infty}, {\tilde m}^{\infty})$.
\end{lemma}

An equivalent result was proved by J.~F.~C.~Kingman \cite{King75}.
The theorem which says that
the projection of the gamma measure $\Lambda^{\theta}$ to the simplex
is the Poisson--Dirichlet measure $PD(\theta)$ is also presented in
\cite{Arr}. In order to prove this fact (see \cite{TVY}), it suffices to
check that the measure $PD(\theta)$ has the required system of
self-consistent distributions, i.e., the same
system of distributions as $D^{\theta}$; this is an easy
consequence of the definitions of $PD(\theta)$ and the Dirichlet
distributions.\footnote{The main difficulty is to prove that the
gamma measure is concentrated on the space $l^1(X)$, but this is a classical
fact.}

\begin{theorem}
Denote by $(X^{\infty}, \tilde m^{\infty})$ the infinite product of the
measure spaces $(X,\tilde m)$ (here $\tilde m$ is the normalization of the measure
$m$). With respect to the decomposition $l^1_+(X)\setminus
\{{\bf 0}\}=s^1(X)\times (0,\infty)$, the $\pmod 0$ bijection $J$ is
a measure-preserving homomorphism which represents the measures
$$D^{\theta},\, \Lambda^{\theta},\, {\cal L}^{\theta}$$ as follows:
$$ D^{\theta}=PD(\theta)\times \tilde m^{\infty};$$
$$\Lambda^{\theta}=PD(\theta)\times \tilde m^{\infty}\times \lambda^{\theta};$$
$${\cal L}^{\theta}=PD(\theta)\times \tilde m^{\infty}\times L_{\theta}=
CPD(\theta)\times \tilde m^{\infty},$$ where $CPD(\theta)=PD(\theta)\times
L_{\theta}$ is the {\it conic Poisson--Dirichlet measure}. In
particular, for the infinite-dimensional Lebesgue measure ${\cal L}$ we
have
$${\cal L}=CPD(1)\times \tilde m^{\infty}.$$
\end{theorem}

Consider the cone of monotonic positive convergent series
$${\bar l}_+^1=\{\{c_k\}_{k=1}^{\infty}: c_1\geq c_2\geq\dots \geq 0, \; \sum_k
c_k<\infty \};$$
in a natural sense, this is the product
 of the simplex $\bar \Sigma$ and the half-line, together with the vertex $\bf 0$:
 ${\bar l}_+^1=(\bar \Sigma\times (0,\infty)) \bigcup \{\textbf{0}\}$.
 We have defined the measure $CPD(\theta)$ on the space ${\bar l}_+^1$, or the
conic superstructure of the Poisson--Dirichlet measure, as
$$CPD(\theta)=PD(\theta)\times L_{\theta},$$ where $PD(\theta)$ is a
measure on the simplex $\bar \Sigma$ and the measure $L_{\theta}$ on
$(0,\infty)$ was defined as $dL_{\theta}=\frac
{x^{\theta-1}}{\Gamma(\theta)}dx$. As we will see, the measure $CPD(\theta)$ has
more symmetries than $PD(\theta)$.
It seems that the $\sigma$-finite conic Poisson--Dirichlet measure $CPD(1)$ on
the cone ${\bar l}_+^1$ (in other words, the  conic
superstructure) has not been considered in the literature.
This is a very interesting measure, which has many natural characterizations.

\subsection{The Laplace transform of the measure ${\cal L}^{\theta}$}

The dual method of defining measures in a linear
space consists in using the Fourier or Laplace transforms.
In the case of a probability measure $\mu$ on a linear space
$E$, for every linear continuous (and even linear measurable) functional
$f:E \rightarrow{\Bbb R}$ we can consider its Fourier transform
$$\int_E e^{<f,\xi>}d\mu(\xi)\equiv \Psi_{\mu}(f)$$
 with respect to $\mu$, and thus obtain
a positive definite functional $\Psi_{\mu}$ on the dual space,
which uniquely determines the measure $\mu$. In the case of a $\sigma$-finite
measure, some or all nonzero linear functionals may have no finite
distributions, so we cannot use this method. But, fortunately, in our case there is a total set of linear
functionals that have finite distributions, and we can define the
Laplace transforms of these functionals.

\begin{theorem}
Assume that a test function $f$ on the space $X$ is positive ($f(x)>0$)
and satisfies the condition
$$\int_X \ln f(x)dm(x)<\infty.$$
Then
$$\int_{\l^1_+(X)} e^{<f,\xi>}d{\cal L}^{\theta}(\xi)\equiv \Psi_{\theta}(f)=
e^{-\theta\int_X \ln f(x)dm(x)}.$$ The functional $\Psi_{\theta}$
uniquely determines the measure, i.e., there is only one measure, ${\cal
L}^{\theta}$, with this Laplace transform.
\end{theorem}

\begin{proof}
Let us begin with positive step functions of the type $f=\sum_{i=1}^n
d_i\chi_{E_i}$, $d_i>0$, with equal measures of steps:
$m(E_i)=\frac{\theta}{n}$, $\theta=m(X)$; put $\bar
\theta=(\frac{\theta}{n},\dots ,\frac{\theta}{n})$. Then, using the
definition of the finite-dimensional projections of ${\cal
L}^{\theta}$, we obtain
$$\int_{{\Bbb R}^n_+}e^{-\sum <f,\xi>}dL_{\bar \theta}=
\prod_i \int e^{-d_ix_i} dL_{\theta}(x_i)$$ $$=\prod_i
d_i^{-\frac{\theta}{n}}=e^{-\frac{\theta}{n}\sum_i \ln d_i}=
e^{-\theta \int_X\ln f(x)dm(x)}.$$ Then this formula
can be extended by continuity to all positive functions with
$\int_X \ln f(x)dm(x)<\infty$.
The uniqueness follows from the totality of the set of positive functions.
\end{proof}

The Laplace transform of the Lebesgue measure $\cal L $ is given by
the formula
$$\Psi_{\cal
L}(f)\equiv \Psi (f) =e^{-\int \ln f(x)dm(x)}, \quad m(X)=\theta=1.$$
This formula can be
taken as a definition of the measure $\cal L$: this is the unique $\sigma$-finite
measure whose Laplace transform is the functional $\Psi$ (see \cite{V}).

\medskip\noindent
{\bf Remark.} What are the distributions of linear functionals? It is
easy to check that if a function is nonpositive on a set of
positive measure, then the corresponding linear functional
on the space $l^1_+(X)$
has no finite distribution, because the preimage of every measurable set
(i.e., every Lebesgue set) on the real line is of infinite measure.
If the condition $\int \ln
f(x)dm(x)=c(f)<\infty$ is satisfied, then the functional $<f,\cdot>$ has
a proper distribution with respect to the measure ${\cal L}^{\theta}$,
which is simply the measure $e^{c(f)}\cdot L_{\theta}$. In particular,
the distribution of the functional $<f,\cdot>$
with respect to the Lebesgue measure
$\cal L$ is the Lebesgue measure $e^{c(f)}\cdot L$ on the half-line.

  This is a big contrast with the case of the gamma measure and
  Dirichlet measure, for which the distribution of the linear functional on $D(X)$
  generated by a function $f$ can be very various, and the relation
  between this distribution and the distribution of $f$ as a function on $X$ is
  very nontrivial (this question is related to the so-called
  Markov--Krein transform, see \cite{Ker}).

\subsection{The invariance property of the measures ${\cal
L}^{\theta}$}

 The main property of the measures ${\cal L}^{\theta}$, $\theta>0$,
and, in particular, of the Lebesgue measure is the multiplicative
invariance.

\begin{definition}
Denote by ${\cal M}(X)$ and ${\cal M}_0(X)$ the following groups, with respect to pointwise
multiplication, of
continuous bounded functions on $(X,m)$:
$${\cal M}_0(X)=\{a:a(x)\geq 0, \; \int_X \ln a(x)dm(x)=0\}$$
and
$${\cal M}(X)=\{a:a(x)\geq 0 \quad |\int_X \ln
a(x)dm(x)| < \infty\}.$$
The topology is defined by the following system of neighborhoods
of the identity: $V_{\epsilon}({\bf 1})= \{a\in {\cal M}: \int
|a(x)-1|dm<\epsilon<1\}$.\footnote{In this framework, one may
also consider the set of measurable bounded functions satisfying the
same condition, but here we restrict ourselves to continuous
functions.}
\end{definition}

Denote $$\phi(a)=e^{\int_X \ln a(x)dm(x)}.$$

We call $\cal M$ the {\it continual Cartan group}, and ${\cal M}_0$,
the special continual Cartan group. These groups are natural
 continual versions of the groups $SL(n,{\Bbb R}_+)$ and $GL(n,{\Bbb R}_+)$.
 The group ${\cal M}(X)$ acts on the space $l^1(X)$, as well as on the space
 of test functions. Sometimes we denote by $M_a$
 the operator of multiplication by a
 function $a$: $M_a f=a\cdot f$.

 Consider the group ${\frak A}(X)$ of measurable transformations of
 the space $(X,m)$. This group acts on the space $D(X) \supset
 l^1_+(X)$ as $(U_T f)(\cdot)=f(T\cdot)$ So $U_T$ is a linear
 operator that acts in the spaces $D(X)$ and $l^1_+(X)$. The group
 ${\frak A}(X)$ can be regarded as an analog of the Weyl group in the
 group $GL(n,\Bbb R)$. We can also define the cross product of the
 groups  ${\frak A}(X)$ and $\cal M$.

 \begin{theorem}[Invariance]
 Let $a \in \cal M$.

 {\rm 1.} The measures ${\cal L}^{\theta}$ are
 invariant, up to a multiplicative constant, under the action of the
 group $\cal M$:
 $${\cal L}^{\theta}(M_a^{-1} A)=\phi(a){\cal L}^{\theta}(A),\quad a \in \cal M, $$
 for any measurable set $A \subset l^1_+(X)$. In particular, the
 measures ${\cal L}^{\theta}$ are invariant under the action of the group
 ${\cal M}_0$:
 $${\cal L}^{\theta}(M_a^{-1} A)={\cal L}^{\theta}(A),\quad a \in {\cal M}_0.$$

{ \rm 2.} The measures ${\cal L}^{\theta}$ are invariant under the
 action of the group ${\frak A}(X)$.
 \end{theorem}

 \begin{proof}
 1. Let us consider the behavior of the Laplace transform
 of the measure ${\cal L}^{\theta}$ under the action of
 the group $\cal M$. Denote for a moment ${\cal L}'(\cdot)= {\cal
L}^{\theta}(M_a \cdot)$, $a \in \cal M$. Then
 $$\Psi_{{\cal L}'}(f)=\int e^{-<\eta,f>}d{\cal L}^{\theta}(M_a \eta)=\int
e^{-\theta <M_{a^{-1}}\eta, f>}d{\cal L}^{\theta}(\eta)$$
$$=\int e^{-\theta <\eta,
M_{a^{-1}} f>} d{\cal L}^{\theta}(\eta)=e^{-\theta \int_X \ln
[a(x)^{-1}f(x)]dm(x)}$$ $$=e^{-\theta \int_X [\ln f(x)-\ln
a(x)]dm(x)}={\phi(a)}^{\theta}\cdot\Psi_{{\cal L}^{\theta}}(f).$$
By the uniqueness theorem for the Laplace transform, we have
$${\cal L}^{\theta}(M_a \cdot)=\phi(a)^{\theta}{\cal L}^{\theta}(\cdot).$$
So if $a\in {\cal M}_0$ then $\phi(a)=1$, and the measure ${\cal
L}^{\theta}$ is invariant under the operator $M_a$. For $a \in \cal
M$, the measures ${\cal
L}^{\theta}$ are {\it projective invariant} under the operator
$M_a$.
 Below we will prove that
the action of the measure-preserving group ${\cal M}_0$ is ergodic.

2. The second claim follows from the fact that the Laplace transform $\Psi_{{\cal L}^{\theta}}(f)$ of
the measure ${\cal L}^{\theta}$ depends only on the
distribution of the function $f$, which does not change under
measure-preserving transformations.
\end{proof}

\medskip\noindent
{\bf Remark.} If $a\equiv {\rm const}>0$, then $\phi(a)=a$, and we
have ${\cal L}^{\theta}(a \cdot)=a^{\theta}{\cal L}^{\theta}(\cdot)$. For
the Lebesgue measure ($\theta=1$), we see that the measure $\cal L$
is homogeneous of degree one. This is an important property: the degree of
homogeneity of the $n$-dimensional Lebesgue measure is equal to $n$, and it is
natural to believe that in the infinite-dimensional case there is
infinite homogeneity; but our measure ${\cal L}$ has homogeneity of degree one!

\subsection{Uniqueness theorem for the measures ${\cal L}^{\theta}$}

We will prove the uniqueness of the family of measures ${\cal
L}^{\theta}$ that are invariant under the action of the group ${\cal
M}_0$, ergodic, and subject to some conditions. Recall that
${\cal M}_0$ is the multiplicative group of nonnegative
functions on the interval,
$${\cal M}_0=\{a: \int_0^1\ln a(t)dm(t)=0\},$$ which acts on the space of
test functions, as well as on the space of Schwartz distributions, as a group of
multiplicators.

\begin{theorem}
Let $\cal L$ be a $\sigma$-finite Borel measure on the cone
$l^1_+(X)$ whose Laplace transform $\Psi_{\cal L}$ is finite on the cone of
positive test functions and continuous on this cone:
$$\int_{l^1_+(X)}\exp\{-<f,\eta>\}d{\cal L}(\eta)\equiv \Psi_{\cal L}(f)\equiv \Psi(f)<\infty$$
for all positive bounded test functions $f$. Suppose that the
measure $\cal L$ is invariant under the action of the group ${\cal
M}_0$ and is homogeneous of degree $\theta$ under the multiplication by a
constant:
    ${\cal L}(cE)=c^{\theta}{\cal L}(E)$.

Then ${\cal L}={\cal L}^{\theta}$.
\end{theorem}

\begin{proof} In \cite[Theorem 4.2, p.~285]{TVY} it
 was proved that a measure that is absolutely continuous with respect to the
 law of a L\'evy process and satisfies the multiplicative invariance is one
 of the measures ${\cal L}^{\theta}$. We prove that instead of the
 absolute continuity with respect to the law of a L\'evy process it
 suffices to assume the multiplicativity. The first step is to prove
 that the invariance of the measure $\cal L$ under the group of
 measure-preserving transformations follows from the invariance and
 projective invariance under the groups ${\cal M}_0$ and $\cal M$.
 The invariance under the action of the group ${\cal M}_0$ means that for
 every function $a \in {\cal M}_0$ we have $\Psi(a\cdot f)=\Psi(f)$
 for all positive test functions $f$. If we choose $f\equiv 1$, then
 $\Psi(a)=\Psi(1)<\infty$. This means that the functional $\Psi$ takes
 the same values on the whole group ${\cal M}_0$, and since the
 measure is $\sigma$-finite, we can normalize it so that $\Psi(1)=1$.
 But $\Psi(c\cdot a)=c^{\theta}\Psi(a)=c^{\theta}\Psi(1)=c^{\theta}$,
 so that $\Psi$ is a homomorphism of the group $\cal M$ to the group of positive
 numbers: $\Psi:{\cal M}\rightarrow {\cal M}/{\cal M}_0={\Bbb R}_+$.
 Consequently, taking into account the normalization, we obtain
 $\Psi(f)=\exp\{-\int_X \theta \ln a(x)dm(x)\}$. Recall once again
 that the measures we had constructed have finite degrees of
 homogeneity; in the case of the Lebesgue measure, it is equal to $1$.
\end{proof}

The invariance of the
measure ${\cal L}^\theta$
under the group ${\cal M}_0$ and the uniqueness theorem above, together with
the relation between ${\cal L}^\theta$ and
$CPD(\theta)$, imply the following very important
characteristic property of the measure $CPD(\theta)$:

\begin{corollary}[see also \cite{V}]
Consider the conic Poisson--Dirichlet measure $CPD(\theta)$ on the
cone ${\bar l}^1$ of monotonic positive convergent series $\{x_k\}$.
Suppose that, given a probability vector ${\bar \theta}=(\theta_1,\theta_2,
\dots, \theta_n)$, we divide the set of elements
of the random series $\{x_k\}$ into $n$ parts, assuming that
each element independently belongs to the $i$th part with
probability $\theta_i$. Calculating the sum of each part, we
obtain $n$ numbers. Then the joint distribution of these $n$ sums is
the distribution $ L_{\bar \theta}$ in ${\Bbb R}^n_+$ (see Section~{\rm2.4}).
Conversely, if some measure on the cone ${\bar l}^1$ satisfies this
property, then this measure coincides with $CPD(\theta)$.
\end{corollary}

\medskip\noindent
\textbf{Question}. Assume that a bounded sequence of positive
numbers $a_1, a_2, \dots$ has the following property:
$$\lim_{n \to \infty} \frac{1}{n}\sum_{k=1}^n \ln a_k = 0.$$ Does
the map $\{c_k\}\mapsto \{a_k\cdot c_k\}$ preserve the measures
$CPD(\theta)$ for all $\theta$?

\medskip
The positive answer on this question is a generalization of the
claim of the previous theorem.

 \subsection{Ergodicity of the action of the group ${\cal M}_0$}

On first sight, the action of the group ${\cal M}_0$ on the cone
${\bar l}^1_+(X)$ does not change $x_k$, so that it is not
ergodic. But this is not true.

\begin{theorem}
The measure-preserving action of the group ${\cal M}_0$ on the measure space
$(l^1_+(X), {\cal L}^{\theta})$ is ergodic.
\end{theorem}
\begin{proof}
The uniqueness theorem for the measures ${\cal L}^{\theta}$
already contains the ergodicity; indeed, if the measure is not ergodic, then
we can decompose it into ergodic components,\footnote{Here we can do
this, because the action of the group $\cal
M$ on $l^1_+(X)$ is measurable (individual). It is important to
mention this fact, because the group is not locally compact.}
so there is no uniqueness. But we will give a sketch of a
direct proof, because it uses the structure of the action. Let us apply the
above-defined isomorphism $J$ between the space $(l^1_+(X),
{\cal L}^{\theta})$ and the product space  $({\bar l}^1
\times X^{\infty}, CPD(\theta)\times \tilde m^{\infty})$. Then the
action of the group ${\cal M}_0$ is given by the following formula:
$$M_a(\{c_k\},\{x_k\})=(\{a(x_{g(k)})c_{g(k)}\}, \{x_{g(k)}\})$$
where $g=g_{a,\{x_k\}}$ is the permutation of positive integers
that arranges the sequence $\{a(x_k)c_k\}$ by decreasing:
$a(x_{g(1)})\cdot c_{g(1)}\geq a(x_{g(2))}\cdot c_{g(2)}\geq\dots $. The
action of the group  ${\cal M}_0$ is not a skew product, but is similar to
it. Namely, if we drop the monotonic reordering of the sequence
$\{x_k\}$, then the action of the group ${\cal M}_0$ becomes
fiberwise (i.e., it does not change the sequence $\{x_k\}$) and ergodic in
the fibers. But the monotonic reordering of the sequence $\{x_k\}$
is an ergodic action on the Bernoulli product $X^{\infty}$; from this
we can conclude that the action of ${\cal M}_0$  is  ergodic.
\end{proof}

Some details about the measures ${\cal L}^{\theta}$ and their links
with $PD(\theta)$ and $CPD(\theta)$ can be found in \cite{V}.

\section{The absence of geometric approximations of the
Lebesgue measures}

\subsection{About approximations. The case of the Gaussian measure}
 A measure in an infinite-dimensional space
 that has a big group $G$ of symmetries, like the measure
 ${\cal L}^{\theta}$, can be represented
 as the weak limit of finite-dimensional invariant measures.
 Suppose we can find a dense subgroup of $G$ that is the union of an
 increasing sequence of finite-dimensional
 subgroups $G_n$ of the group $G$. Then
 it is natural to describe the measure as the weak limit of invariant measures
 on the orbits of $G_n$. This method of description of invariant
measures of ``big'' groups was called the ``ergodic method'' (see \cite{V1}),
because, in a sense, it uses various types of individual
ergodic theorems. For example, for the Gaussian measure in the
infinite-dimensional linear space, the group of symmetries is the
infinite-dimensional orthogonal
group, and, according to the remarkable
Maxwell--Poincar\'e lemma, this measure is the weak limit of  invariant measures
on the finite-dimensional spheres of increasing radii (see below; for
details, see also  \cite{V}). This also proves the Schoenberg theorem
on description of $O(\infty)$-invariant measures (see \cite{Ak}).

One of the definitions of the white noise, regarded as a generalized process
in the space $L^2(X,m)$, is in the framework of the theory of L\'evy
processes: the one-dimensional Gaussian measure is infinitely divisible,
and the corresponding generalized process is exactly the white noise.
This is parallel to the definitions of the gamma process and the ``$\sigma$-finite
L\'evy process'' from Section~2. We can also formulate this
definition using a ``random Gaussian measure.'' Another equivalent
(dual) definition uses the characteristic functional, or Fourier
transform. This notion makes sense not only for true $\sigma$-additive
Borel measures, but also for cylinder measures, or
generalized processes, which are defined only on the algebra of
cylinder sets of a linear topological space with measure, or
cylinder measure, $\mu$:
  $$\Psi_{\mu}(f)=\int_{-\infty}^{+\infty} \exp\{it\}d\mu_f(t);$$
  here $\mu_f$ is the distribution of the functional $f$ with
  respect to the (cylinder) measure $\mu$.
The characteristic functional is defined on the dual (conjugate)
space of linear functionals on the space where the measure is
defined. In the case of a Hilbert space, both spaces coincide. The
white noise is defined as a true measure in the Hilbert--Schmidt
extension of the space $L^2(X,m)$, but we do not need to consider this space.
The white noise $W$ is the generalized process in the space $L^2(X,m)$ (or
in an arbitrary Hilbert space) with characteristic functional
     $$\Psi_W(f)=\exp\{-\|f\|^2\}.$$

Let $\xi_n \equiv \xi$ be an arbitrary partition of the space $X$ into
$n$ pieces $A_1, \dots, A_n$ of equal measure $m(A_i)=1/n$, $i=1,
\dots, n$. Denote by $E_n$ the $n$-dimensional subspace of $L^2(X,m)$ that consists of
all functions that are constant on all $A_i$, $i=1, \dots, n$, with the
induced norm. Denote by $\rho_{\xi}$ the normalized Lebesgue measure on
the $(n-1)$-sphere $E_n\bigcap S_1$, where $S_1=\{f\in H: \|f\|=1\}$
is the unit sphere in $H$. Obviously, we can regard $\rho_{\xi}$ as
a cylinder measure.

\begin{theorem} The limit of the sequence of characteristic
functionals of the cylinder measures $\rho_{\xi_n}$ is
the characteristic functional of the white noise.
\end{theorem}

We will not go into the details, but a consequence of this theorem is
as follows.

\begin{corollary} The measure generated by the white noise (which is
 defined in the Hilbert--Schmidt extension of $L^2(X,m)$) is the weak
 limit of the finite-dimensional measures $\rho_{\xi_n}$.
 \end{corollary}

\begin{proof}We have
$$\lim_n \int_{O_n}\exp\{i<s,x>\}d\omega_n(x)=\lim_n
C_n\bigg[\int_0^1(1-r^2)^{\frac{n-3}{2}}\cos(\|s\|\sqrt n
r)dr\bigg]$$
$$=\lim_n \Gamma\bigg(\frac{n}{2}\bigg)\bigg(\frac{2}{\|s\|\sqrt n}\bigg)^{\frac{n-3}{2}}
\cdot J_{\frac{n-3}{2}}(\|s\|\sqrt n)=\lim_n
\bigg[-\frac{\|s\|^2}{2}+\frac{\|s\|^4}{2\cdot 4(1+2/n)}-\dots \bigg]
$$
$$=\exp(-\frac{\|s\|^2}{2}),$$
where $s \in {\Bbb R}^n$; $O_n$ is the $(n-1)$-sphere of radius
 $\sqrt n$; $x \in O_n$; $\omega$ is the normalized Lebesgue measure on
 the sphere $O_n$; and $J_m(\cdot)$ is the Bessel function. Here we have
 used the
 standard asymptotics of the Bessel function $J_m$ and the gamma function, see
\cite{GR,Ak,V}.
\end{proof}

\medskip\noindent
{\bf Remarks. 1.}
The calculations above are equivalent to the
proof of the Maxwell--Poincar\'e (MP) lemma, which claims that the
weak limit of the normalized Lebesgue measures on the spheres $S^n
\subset \mathbb R^{n+1}\subset\mathbb R^{\infty}$ of radius $\sqrt
n$ is the standard Gaussian measure in the space $\mathbb
R^{\infty}$. See \cite{TV} for details and history. But the
conclusion of our theorem involves another space rather than
$\mathbb R^{\infty}$, and approximation in our case is quite
different.

\smallskip
{\bf2.} Theorems that claim that an infinite-dimensional measure
with an infinite-dimensional group of symmetries is the weak limit
of measures on the orbits of finite-dimensional groups can be
regarded as claims about the equivalence of the small and grand
canonical ensembles in statistical physics. In the case  of the
Gaussian measure, this analogy can be pursued further.

\subsection{The case of the infinite-dimensional Lebesgue measure}

 Is it possible to approximate our measures ${\cal L}^{\theta}$ in
the same way using the orbits of the finite-dimensional Cartan
subgroups? The similarity between the Gaussian and Lebesgue cases is
obvious: both measures are obtained by the same construction of
L\'evy processes, the first one being generated by the semigroup of Gaussian
measures on the line with density
$G(x)=\frac{1}{\sqrt{2\pi}}e^{-\frac{x^2}{2}}$, and the second one,
by the semigroup of $\sigma$-finite measures with densities
$L_{\theta}(x)=\frac{x^{\theta-1}}{\Gamma(\theta)}$ on the half-line.
Nevertheless,  the difference between the
infinite-dimensional Gaussian and Lebesgue measures is also big:
the group of symmetries ${\cal M}_0$ is abelian and  does
not contain a dense subgroup that is the limit of compact groups,
as in the case of the orthogonal group $O(\infty)$. But it
 has many finite-dimensional noncompact subgroups.
So we can consider the
orbits of those subgroups that are smooth noncompact manifolds in
the cone $\l^1_+(X)$.

In comparison with the Gaussian case, instead of the $n$-dimensional spheres
$S^n_{r_n}$ of radius $r_n=c\sqrt n$ used in the Maxwell--Poincar\'e lemma,
we must consider a hypersurface, the ``hypersphere'' in ${\Bbb R}^n$ defined as
$$M_{n,r} = \bigg\{(y_1, \dots, y_n): \prod_{k=1}^n y_k =r^n,\, y_k>0,\, k=1, \dots, n \bigg\}.
$$
The number $r=r(n)$ will be called the \emph{radius} of the
 hypersphere; it depends on $n$.
This hypersphere is a homogeneous space of the positive part $SD_+(n)$ of the Cartan subgroup
of $SL(n,\Bbb R)$. Now let $\mu_{n,r}$ be the invariant
$\sigma$-finite measure on $M_{n,r}$ with some scaling (= a choice of a
set of measure $1$).

Now we formulate a precise statement which shows that in the Lebesgue case
there is no
approximation of this type. A detailed version
can be found in \cite{VV}.

We want to find the asymptotic properties of the invariant measure
on the positive Cartan subgroup $SDiag_+$ (consisting of positive diagonal
matrices) of the group $SL(N,\mathbb R)$ as $N$ tends to infinity. More
exactly, we want to find the {\it Laplace transform $D_n(\cdot)$ of the
invariant $\sigma$-finite measure on the hypersphere} $M_{n,r}$:
$$D_n(f)=\int _{M_{n,r}}
\exp\{-\sum_{k=1}^n y_k \cdot f_k\}dm_n(y),$$
where $f=(f_1,\dots, f_n)$ are the dual variables. We introduce the function
$F_n$ on the positive half-line ${\Bbb R}_+$ which is sometimes
called the Mellin--Barnes function (see, e.g.,
\cite{Par}):
 $$F_n(\lambda)=\int_{H_n}\exp \{-\lambda \sum_{k=1}^n \exp x_k\}dx,$$
 where
 $H_n=\{(x_1,\dots, x_n)\in {\Bbb R}^n:\sum_{k=1}^n x_k=0\}$.

After some transformations we obtain the formula
$$D_n(f)=F(\rho_n(f)r_n),$$
where
 $\rho_n(f)=(\prod_{k=1}^n
f_k)^{\frac{1}{n}}$.

The function $F_n$ is the inverse Mellin transform of the $n$th power
$\Gamma(s)^n$ of the gamma function  (up to the multiplier $n$, which we can omit).
 Thus the functions $F_n(\lambda)$ and
 $\Gamma(s)^n$ represent a ``Mellin pair'' \cite{Par}.
 For example, for $n=1$ the Mellin pair is $\exp s$ and $\Gamma (s)$.

Thus our problem reduces to the calculation, by the saddle point method,
of the integral
$$F_n(\lambda)=\frac{1}{2\pi in}\int_{\gamma-i\infty}^{\gamma+i\infty} [\Gamma(s)]^n \lambda^{-ns}ds,$$
or, in the ``real'' form (put $s=\gamma + it$),
 $$F_n(\lambda)=\frac{1}{2\pi in}\int_{-\infty}^{+\infty} [\Gamma(\gamma+it)]^n \lambda^{-n(\gamma+it)}dt.$$

 We want  to find an appropriate limit of this sequence as $n \to
 \infty$. It happens that a suitable saddle point is $(\gamma,0)
 \in \Bbb C$ where $\gamma$ is the root of the following equation:
  $$\ln \lambda=\frac{\Gamma'(\gamma)}{\Gamma(\gamma)},$$
 or
  $$\lambda=\exp\bigg\{\frac{\Gamma'(\gamma)}{\Gamma(\gamma)}\bigg\}.$$
The relation between these parameters is expressed by
the following graph:

\includegraphics[width=10cm,height=14cm,angle=270]{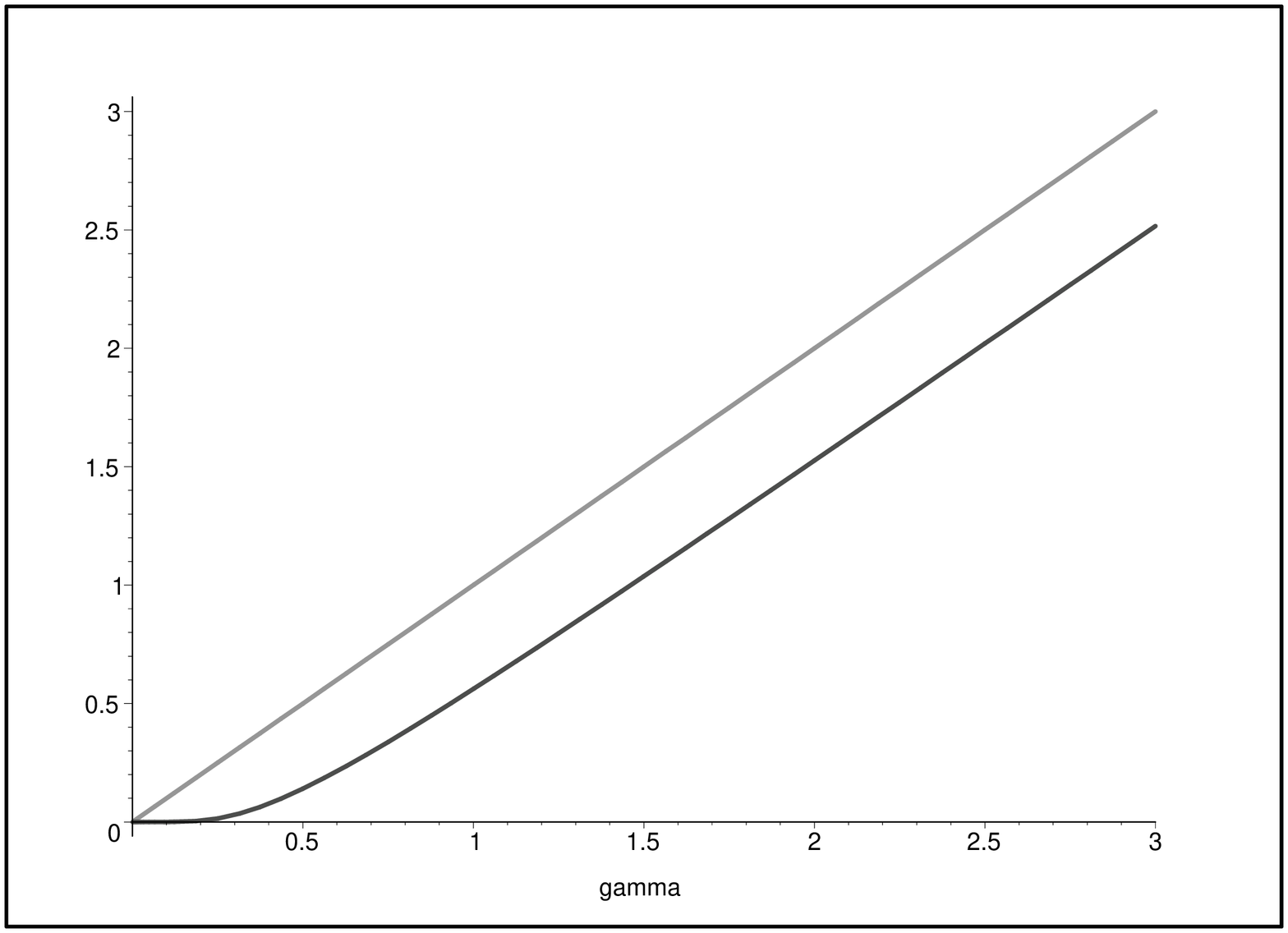}

\bigskip

Consider the function
 $$L(\lambda)\equiv \lim_n\frac{\ln F_n(\lambda)}{n}.$$

The calculation presented in detail in
\cite{VV} yields the following result:

 \begin{theorem}
 The function $L$ satisfies the following equation:
 $$L(\lambda)\equiv \lim_n\frac{\ln F_n(\lambda)}{n}=
 \frac{\Gamma(\gamma(\lambda))}{\lambda^{\gamma}},$$ where $\gamma$
 and $\lambda$  satisfy the equation
 $$\lambda=\exp\bigg\{\frac{\Gamma'(\gamma)}{\Gamma(\gamma)}\bigg\}.$$
 \end{theorem}

This answer proves the existence of the remarkable function $L$ and
means that there is no convergence of the functions $F_n(\lambda)$
as $n\to \infty $ and no convergence of the Laplace transforms $D_n$
({\it for any choice of the radius} $r_n$) of the invariant measures
on the hyperspheres $M_{n,r(n)}$ to a finite limit. Thus $F_n$ does
not tend to the value of the Laplace transform of the measure ${\cal
L}^{\theta}$  at the constant function identically equal to
$\lambda$, which is finite and equal to
$\frac{1}{\lambda^{\theta}}$. Consequently, there is no weak
convergence of the measures $m_n$ (in the sense of convergence of
Laplace transforms). So we cannot represent the measure ${\cal
L}^{\theta}$ as the weak limit of finite-dimensional
$SDiag_+$-invariant measures.  In turn, we can say that there is no
equivalence of the grand and small canonical ensembles for the
Cartan subgroups. It seems that the ``tail'' of the hypersphere
$M_{n,r}$ carries too much part of the measure in order to preserve
the finiteness of the limit and violate the equivalence of the
ensembles.

The function $L$ is a very interesting object; author does not know
if it has been studied in the literature. It looks like the free
energy in statistical mechanics and plays the role of the generator
of the family of invariant measures on the hyperspheres. The graph
of $L$ shows that it is similar to the graph of the function $-\ln$,
which is the true generator of the one-parameter semigroup of
one-dimensional Lebesgue measures ($L_{\theta}$):

\bigskip

\includegraphics[width=14cm,height=10cm,angle=0]{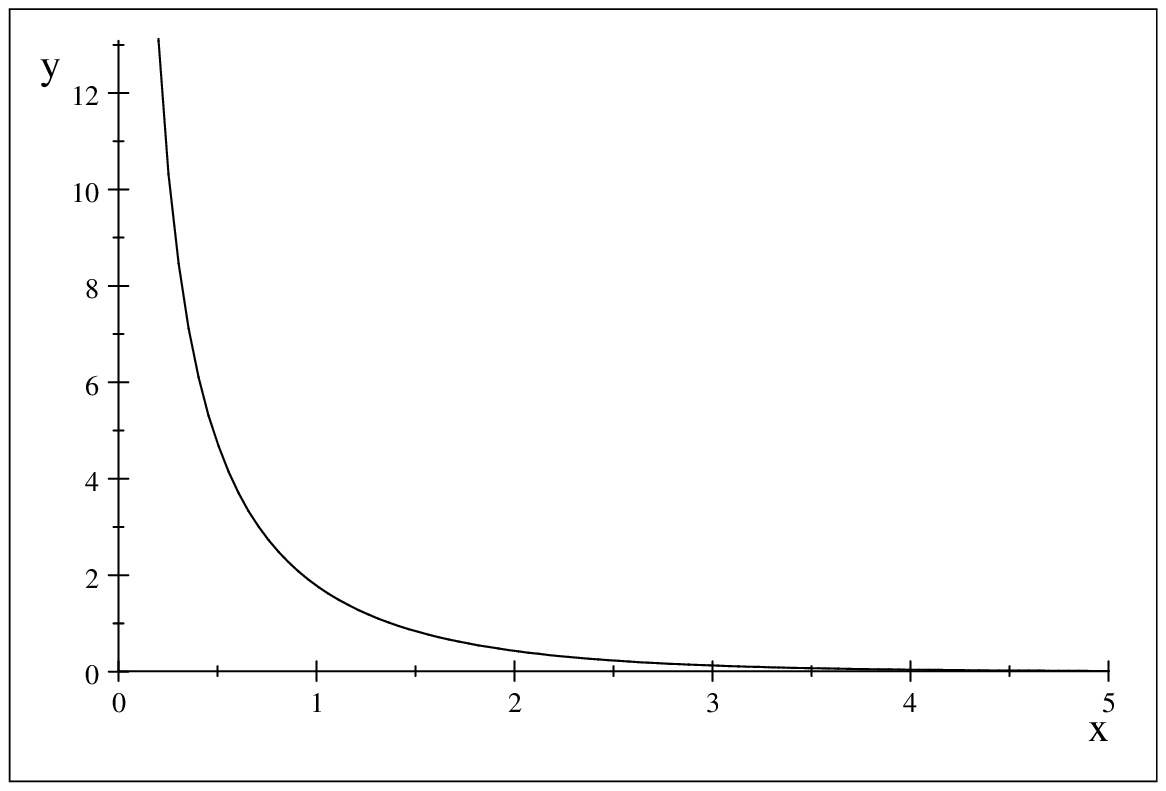}

\bigskip

\pagebreak
\noindent
{\bf Concluding remarks.}

\smallskip
{\bf1.} It is very easy to construct the family of Lebesgue measures on
the whole space $l^1(X)$ instead of the measures ${\cal L}^{\theta}$ on
the cone $l^1_+(X)$: we need only to replace the
gamma subordinator in the definition of ${\cal L}^{\theta}$ with the whole
(nonpositive) gamma process. The measures $L_{\theta}$ on the
half-line must be replaced in all definitions with the measure on the
whole line that is the symmetric extension of
$L_{\theta}$. Apparently, the results of the
theory of Poisson--Dirichlet measures, including those of \cite{VSh1}, have never
been generalized to this case of nonpositive series.

\smallskip
{\bf2.} Another aspect of the theory we discussed in this paper concerns
generalization of the Lebesgue measure to the set of {\it
two-sided series}, which will be considered elsewhere.

\bigskip\noindent{\bf Acknowledgements.}
This work was partially supported by the grants RFBR 08-01-00379 and
NSh-2460.2008.1. The author is grateful to the Erwin Schrodinger
Institute (Vienna) for the possibilities to finish this paper during
the semester ``Statistical Physics and Asymptotic Combinatorics'';
and to Natalia Tsilevich for her great help with  editing this
manuscript.

\end{document}